\documentclass{amsart}
\usepackage{amsmath, amscd, amsthm}
\usepackage{latexsym, amssymb, amsfonts}
\usepackage{url}
\usepackage{color}
\usepackage{graphicx}
\usepackage{hyperref}
\hypersetup{
    colorlinks=false,
    linkcolor=blue,
    filecolor=magenta,      
    urlcolor=cyan,
    pdftitle={Density of subsets of squarefree elements in certain Dedekind domains},
    pdfpagemode=FullScreen,
}

\usepackage{stmaryrd}
\usepackage{geometry}\usepackage{tikz}\usetikzlibrary{positioning}
\usepackage{float}
\usepackage{subcaption}

\numberwithin{equation}{section}

\def\Bigroman{\uppercase\expandafter{\romannumeral\number\count 255 }}
\def\Romannumeral{\afterassignment\BigRoman\count255=}

\DeclareMathOperator{\bbF}{\mathbb{F}}
\DeclareMathOperator{\bbN}{\mathbb{N}}
\DeclareMathOperator{\bbZ}{\mathbb{Z}}
\DeclareMathOperator{\bbQ}{\mathbb{Q}}
\DeclareMathOperator{\bbR}{\mathbb{R}}
\DeclareMathOperator{\bbI}{\mathbb{I}}
\DeclareMathOperator{\bbS}{\mathbb{S}}

\DeclareMathOperator{\cP}{\mathcal{P}}
\DeclareMathOperator{\cO}{\mathcal{O}}

\DeclareMathOperator{\cM}{\mathcal{M}}

\theoremstyle{plain}
\newtheorem{theorem}{Theorem}[section] 
\newtheorem{proposition}[theorem]{Proposition}
\newtheorem{lemma}[theorem]{Lemma}

\theoremstyle{definition}
\newtheorem{definition}[theorem]{Definition}

\theoremstyle{remark}
\newtheorem{remark}[theorem]{Remark}

\title{Density of subsets of squarefree elements in certain Dedekind domains}
\author{Eunju Shin}
\address{Department of Mathematics, The Ohio State University, 231 West 18th Avenue, Columbus, OH 43210-1174}
\email{shin.970@osu.edu}
\usepackage{lipsum}

\begin{document}

\maketitle

\begin{abstract}
    We consider polynomial rings over finite fields and rings of integers of imaginary quadratic fields $\mathbb{Q}(\sqrt{-D})$. In this paper, we formulate the density of squarefree elements divisible by all elements of $T$ but by none of $P$, where $T$ and $P$ are subsets of squarefree elements and $T$ is finite. We also define Mersenne irreducibles in order to estimate the density of squarefree elements divisible by none of $P$.
\end{abstract}

\section{Introduction}
Squarefree numbers have long been among the most studied and significant objects in number theory. It is well-known that the natural density of the set $\bbS_{\bbN}$ of squarefree numbers is
\begin{align*}
    \frac{1}{\zeta(2)}=\frac{6}{\pi^2}\text{\cite{ha}}.
\end{align*}
Jameson showed that the natural density of odd squarefree numbers is
\begin{align*}
    \frac{4}{\pi^2}\text{\cite{j}}.
\end{align*}
Brown generalized this result by obtaining a formula for the natural density of the set of squarefree numbers divisible by all elements of $T$ but by none of $P$:
\begin{align*}
    \frac{6}{\pi^2}\prod_{p\in T}\frac{1}{1+p}\prod_{p\in P}\frac{p}{1+p},
\end{align*}
where $T,P$ are subsets of $\bbS_{\bbN}$ and $T$ is finite\cite{b}.

A natural question is whether similar formulas can be obtained over other rings, especially Dedekind domains.

\subsection{Background and Definitions}
We are concerned with polynomial rings over finite fields and certain rings of integers, both of which are well known examples of Dedekind domains. 

The ring of integers $\cO_{\bbQ(\sqrt{-D})}$ of the imaginary quadratic field $\bbQ(\sqrt{-D})$, with $D\in\bbN$: 
    \begin{align*}
    \cO_{\bbQ(\sqrt{-D})}=\begin{cases}
        \bbZ[\sqrt{-D}]&\text{if $D\not\equiv 3\ (\operatorname{mod}4)$};\\
        \bbZ\left[\frac{1+\sqrt{-D}}{2}\right]&\text{if $D\equiv 3\ (\operatorname{mod}4)$}.
    \end{cases}
\end{align*}
   Up to isomorphism,  $\Bbbk=\bbQ(\sqrt{-D})$ are essentially the only quadratic number fields whose rings of integers have finitely many units. For simplicity, we write $\cO_D=\cO_{\bbQ(\sqrt{-D})}$.

   For a Dedekind domain $R$, we define
\begin{align}\label{norm}
    N:R\to\bbZ_{\geq0},\ N(x):=\begin{cases}
        |R/xR|&\text{if $x\neq0$};\\
        0&\text{if $x=0$},
    \end{cases}
\end{align}
where $|R/xR|$ denotes the number of cosets $r+xR$ in $R$. Then $N$ is multiplicative\cite{nj}. If $R=\bbF_q[t]$, then the map \ref{f} is explicitly given by
\begin{align*}
    N:\bbF_q[t]\to\bbZ_{\geq0},\ N(x)=q^{\deg(x)}.
\end{align*}
If $R=\cO_D$, then the map \ref{f} is explicitly given by
\begin{align*}
    N:\cO_D\to\bbZ_{\geq0},\ N(a+\sqrt{-D}b)=a^2+Db^2.
\end{align*}
\begin{definition}[density] Let $R=\bbF_q[t],\cO_D$ and $A\subseteq R$. For $n\in\bbN$, let 
\begin{align}\label{f}
    F_n=\{x\in R:N(x)<n\}.
\end{align}
The \emph{lower} and \emph{upper densities} of $A$ are defined by
    \begin{align*}
        \underline{d}(A):=\liminf_{n\to\infty}\frac{|A\cap F_n|}{|F_n|}\quad \text{and}\quad
        \overline{d}(A):=\limsup_{n\to\infty}\frac{|A\cap F_n|}{|F_n|}.
    \end{align*}
    If $\underline{d}(A)=\overline{d}(A)$, then $A$ is said to have \emph{density} $d(A)=\underline{d}(A)=\overline{d}(A)$.
\end{definition}
Indeed, $F_n$ forms a F\o lner sequence in $R=\bbF_q[t],\cO_D$, so $\underline{d}$ and $\overline{d}$ are well-defined.

\begin{proposition}\label{sqf}
    One has 
    \begin{align*}
        d(\bbS_{\bbF_q[t]})=\frac{1}{\zeta_{\bbF_q[t]}(2)}&=1-\frac{1}{q}\text{\cite{r}};\\
        d(\bbS_{\cO_D})=\frac{1}{\zeta_{\cO_D}(2)}\text{\cite{cv}},
    \end{align*}
    where $\zeta_R$ is the Dedekind zeta function.
\end{proposition}

\begin{definition}
Let $R$ be a Dedekind domain. Let $a\in R$ be nonzero and nonunit.
\begin{itemize}
    \item[(i)] $a$ is \emph{irreducible} if $a=bc$ implies either $b$ or $c$ is a unit in $R$.

    \item[(ii)] $a$ is \emph{squarefree} if $b^2\nmid a$ (i.e.,\ $a$ does not have $b^2$ as a factor) for any irreducible $b\in R$. 
\end{itemize}
\end{definition}
Let  $\bbS_R=\{p\in R:\text{$p$ is squarefree}\}$ and $\bbI_R=\{p\in R:\text{$p$ is irreducible}\}$. Here, we use $\mathbb{I}_R$, instead of the set of primes, since $R$ may have more than one unit, and primality is not well behaved under multiplication by units.

\subsection{Main theorems}
Henceforth, $R=\mathbb{F}_q[t],\cO_D$.
\begin{theorem}\label{1}
    Let $P,T\subseteq\bbI_R$ be disjoint, with $T$ finite. Then the portion of all in $\bbS_R$ which are divisible by all the irreducibles in $T$ and by none of the irreducibles in $P$ is
\begin{align*}
    d(\bbS_R)\prod_{t\in T}\frac{1}{1+N(t)}\prod_{p\in P}\frac{N(p)}{1+N(p)}.
\end{align*}
\end{theorem}
\begin{theorem}\label{2}
    Let $r,m\in\bbI_R$ be relatively prime. Then the set of squarefree elements divisible by no prime congruent to $r$ modulo $m$ has zero density.
\end{theorem}

\section{Proof of Theorem \ref{1}}
\subsection{Auxiliary Lemmas for Theorem \ref{1}}
Let $T,P\subseteq R$ and $n\in\bbN$. Denote
\begin{align*}
    \bbS_R(T,P)&=\{q\in \bbS_R: t|q\ \forall t\in T\text{ but }p\nmid q\ \forall p\in P\},\\
    \bbS_R(T,P)[n]&=\bbS_R(T,P)\cap F_n.
\end{align*}
\begin{lemma}\label{rf}
    Let $F_n\subseteq\bbF_q[t]$ be as in \ref{f}. For any $k\in\bbZ$ and $n\in\bbN$, one has 
    \begin{align*}
        \frac{|F_{q^kn}|}{|F_n|}=q^k.
    \end{align*}
\end{lemma}
\begin{proof}
    Let $k\in\bbZ$ and $n\in\bbN$.
    \begin{align*}
        \frac{|F_{q^kn}|}{|F_n|}&=\frac{|\{x\in\bbF_q[t]:q^{\deg(x)}<q^kn\}|}{|\{x\in\bbF_q[t]:q^{\deg(x)}<n\}|}=\frac{|\{x\in\bbF_q[t]:\deg(x)<k+\log_qn\}|}{|\{x\in\bbF_q[t]:\deg(x)<\log_qn\}|}=\frac{q^{(k+\lceil\log_qn\rceil)}}{q^{\lceil\log_qn\rceil)}}=q^k.
    \end{align*}
\end{proof}

For any $r\in\bbR_{>0}$, one has (see \cite{c})
\begin{align}\label{gb}
    \pi(r-\sqrt{2})^2\leq|\{(x,y)\in\bbZ^2:x^2+y^2\leq r^2\}|\leq\pi(r+\sqrt{2})^2.
\end{align}
\begin{proposition}\label{la}
    For any $r\in\bbR_{>0}$, one has
    \begin{align}
        \frac{\pi(r-\sqrt{1+D})^2}{\sqrt{D}}\leq |\{x\in\bbZ[\sqrt{-D}]:N(x)\leq r^2\}|\leq \frac{\pi(r+\sqrt{1+D})^2}{\sqrt{D}};\label{n3}\\
        \frac{4\pi(r-\sqrt{1+D})^2}{\sqrt{D}}\leq \left|\left\{x\in\bbZ\left[\frac{1+\sqrt{-D}}{2}\right]:N(x)\leq r^2\right\}\right|\leq \frac{4\pi(r+\sqrt{1+D})^2}{\sqrt{D}}\label{3}.
    \end{align}
\end{proposition}
\begin{proof}
Inequalities \ref{n3} imply \ref{3}, since 
\begin{align*}
\left|\left\{x\in\bbZ\left[\frac{1+\sqrt{-D}}{2}\right]:N(x)\leq r^2\right\}\right|&=\left|\left\{\frac{a}{2}+\frac{\sqrt{-D}b}{2}\in \bbZ\left[\frac{1+\sqrt{-D}}{2}\right]:\frac{a^2+Db^2}{4}\leq r^2\right\}\right|\\
    &=|\{a+\sqrt{-D}b\in\bbZ[\sqrt{-D}]:a^2+Db^2\leq(2r)^2\}|\\
    &=|\{x\in\bbZ[\sqrt{-D}]:N(x)\leq (2r)^2\}|,
\end{align*}
and
\begin{align*}
    (2r+\sqrt{1+D})^2&=4\left(r+\frac{\sqrt{1+D}}{2}\right)^2\leq 4(r+\sqrt{1+D})^2;\\ 
    (2r-\sqrt{1+D})^2&=4\left(r-\frac{\sqrt{1+D}}{2}\right)^2\geq 4(r-\sqrt{1+D})^2.
\end{align*}
Now, consider the ellipse 
\begin{align*}
    E_r=\{(a,b)\in\bbR^2:a^2+Db^2\leq r^2\}.
\end{align*}
For each $m,n\in\bbZ$, let $R_{m,n}=[m,m+1]\times[n,n+1]$. Then $|\{x\in\bbZ[\sqrt{-D}]:N(x)\leq r^2\}|$ equals the number of $R_{m,n}$ whose lower-left corner $(m,n)$ lies inside the ellipse $E_r$.

\noindent If $(m,n)$ lies in the ellipse $E_r$, using the Cauchy-Schwarz,
\begin{align*}
    (|m|+1)^2+D(|n|+1)^2&=m^2+Dn^2+2(|m|+D|n|)+1+D\\
    &\leq r^2+2\sqrt{(m^2+Dn^2)(1+D)}+1+D\\
    &=r^2+2\sqrt{1+D}r+1+D=(r+\sqrt{1+D})^2.
\end{align*}
That is, every vertex of $R_{m,n}$ lies in $E_{r+\sqrt{1+D}}$. 

\noindent Similarly, if $(m,n)$ satisfies $m^2+Dn^2\leq (r-\sqrt{1+D})^2$, then 
\begin{align*}
    (|m|+1)^2+D(|n|+1)^2\leq ((r-\sqrt{1+D})+\sqrt{1+D})^2=r^2,
\end{align*}
by replacing $r$ by $r-\sqrt{1+D}$ in the previous case. That is, every vertex of $R_{m,n}$ lies in $E_r$. Consequently, we obtain
\begin{align*}
    \frac{\pi(r-\sqrt{1+D})^2}{\sqrt{D}}=\operatorname{vol}(E_{r-\sqrt{1+D}})\leq |E_r\cap \bbZ^2|\leq \operatorname{vol}(E_{r+\sqrt{1+D}})=\frac{\pi(r+\sqrt{1+D})^2}{\sqrt{D}},
\end{align*}
yielding the inequalities \ref{n3}.
\end{proof}
\begin{lemma}\label{ro}
    Let $F_n\subseteq\cO_D$ be as in \ref{f}. For any $k\in\bbN$, one has 
    \begin{align*}
        \lim_{n\to\infty}\frac{|F_{kn}|}{|F_n|}=k.
    \end{align*}
\end{lemma}
\begin{proof}
    Let $k\in\bbN$. For $D\not\equiv3\ (\operatorname{mod}4)$, using inequalities \ref{n3} in Proposition \ref{la},
    \begin{align*}
        \frac{\left(\sqrt{k}-\frac{\sqrt{1+D}}{\sqrt{n}}\right)^2}{\left(\sqrt{1+\frac{1}{n}}+\frac{\sqrt{1+D}}{\sqrt{n}}\right)^2}=\frac{\frac{\pi(\sqrt{kn}-\sqrt{1+D})^2}{D}}{\frac{\pi(\sqrt{n+1}+\sqrt{1+D})^2}{D}}\leq \frac{|F_{kn}|}{|F_n|}\leq \frac{\frac{\pi(\sqrt{kn+1}+\sqrt{1+D})^2}{D}}{\frac{\pi(\sqrt{n}-\sqrt{1+D})^2}{D}}=\frac{\left(\sqrt{k+\frac{1}{n}}+\frac{\sqrt{1+D}}{\sqrt{n}}\right)^2}{\left(1-\frac{\sqrt{1+D}}{\sqrt{n}}\right)^2}.
    \end{align*}
    Similarly, for $D\equiv3\ (\operatorname{mod}4)$, using inequalities \ref{3} in Proposition \ref{la},
    \begin{align*}
        \frac{\left(\sqrt{k}-\frac{\sqrt{1+D}}{\sqrt{n}}\right)^2}{\left(\sqrt{1+\frac{1}{n}}+\frac{\sqrt{1+D}}{\sqrt{n}}\right)^2}=\frac{\frac{4\pi(\sqrt{kn}-\sqrt{1+D})^2}{D}}{\frac{4\pi(\sqrt{n+1}+\sqrt{1+D})^2}{D}}\leq \frac{|F_{kn}|}{|F_n|}\leq \frac{\frac{4\pi(\sqrt{kn+1}+\sqrt{1+D})^2}{D}}{\frac{4\pi(\sqrt{n}-\sqrt{1+D})^2}{D}}=\frac{\left(\sqrt{k+\frac{1}{n}}+\frac{\sqrt{1+D}}{\sqrt{n}}\right)^2}{\left(1-\frac{\sqrt{1+D}}{\sqrt{n}}\right)^2}.
    \end{align*}
    Since $\lim_{n\to\infty}\frac{\left(\sqrt{k}-\frac{\sqrt{1+D}}{\sqrt{n}}\right)^2}{\left(\sqrt{1+\frac{1}{n}}+\frac{\sqrt{1+D}}{\sqrt{n}}\right)^2}=\lim_{n\to\infty}\frac{\left(\sqrt{k+\frac{1}{n}}+\frac{\sqrt{1+D}}{\sqrt{n}}\right)^2}{\left(1-\frac{\sqrt{1+D}}{\sqrt{n}}\right)^2}=k$, it follows that  $\lim_{n\to\infty}\frac{|F_{kn}|}{|F_n|}=k$ in either case.
\end{proof}
\begin{lemma}\label{u}
    For any $A\in\cP_f(\bbI_R)$, disjoint from both $T$ and from $P$, and for any $n\in\bbN$,
\begin{align*}
    \bbS_R(T,A\cup P)[n]=\bbS_R(T\cup A,P)[nN(m_A)],
\end{align*}
where $m_A=\prod_{p\in A}p$. Moreover, the set $\bbS_R(T\cup A,P)$ has density if and only if $\bbS_R(T,A\cup P)$ has density, and if these densities exist, then
\begin{align*}
    d(\bbS_R(T,A\cup P))=N(m_A) d(\bbS_R(T\cup A,P)).
\end{align*}
\end{lemma}
\begin{proof}
    Consider a map 
    \begin{align*}
        \sigma:\bbS_R(T,A\cup P)[n]&\to\bbS_R(T\cup A,P)[nN(m_A)],\\
        b&\mapsto bm_A.
    \end{align*}
    Let $b\in \bbS_R(T,A\cup P)[n]$. Then $N(b)<n$, so $N(bm_A)<nN(m_A)$, since $N$ is multiplicative. Thus, $\sigma$ is well-defined and, moreover is injective. Let $c\in \bbS_R(T\cup A,P)[nN(m_A)]$. Then 
    \begin{align*}
        &\text{$q|c$ $\forall q\in T\cup A$ and $p\nmid c$ $\forall p\in P$}\\
        &\Leftrightarrow\text{$t|c$ $\forall t\in T$, $a|c$ $\forall a\in A$, and $p\nmid c$ $\forall p\in P$}\\
        &\Leftrightarrow\text{$t|c$ $\forall t\in T$, $m_A|c$, and $p\nmid c$ $\forall p\in P$}\\
        &\underset{T\cap A=\emptyset}{\Leftrightarrow}\text{$t|(c/m_A)$ $\forall t\in T$ and $p\nmid c$ $\forall p\in P$}.
    \end{align*}
   By multiplicativity, $nN(m_A)>N(c)=N((c/m_A)m_A)=N(c/m_A)N(m_A)$, so $N(c/m_A)<n$. Thus, $c/m_A\in \bbS_R(T,A\cup P)[n]$ and $\sigma(c/m_A)=(c/m_A)m_A=c$. Hence, $\sigma$ is bijective, so that $|\bbS_R(T,A\cup P)[n]|=|\bbS_R(T\cup A,P)[nN(m_A)]|$. This implies 
\begin{align*}
    \frac{|\bbS_R(T,A\cup P)[n]|}{|F_n|}=\frac{|F_{nN(m_A)}|}{|F_n|}\frac{|\bbS_R(T\cup A,P)[nN(m_A)]|}{|F_{nN(m_A)}|},
\end{align*}
and hence $d(\bbS_R(T,A\cup P))=N(m_A)d(\bbS_R(T\cup A,P))$ by Lemmas \ref{rf} and \ref{ro}.
\end{proof}
\begin{lemma}\label{s}
    Let $p\in\bbI_R$ not in $T$. If the set $\bbS_R(T,\emptyset)$ has density $\delta$, then 
    \begin{align*}
        d(\bbS_R(T\cup\{p\},\emptyset))=\frac{1}{N(p)+1}\delta.
    \end{align*}
\end{lemma}
\begin{proof}
    For each $n,m\in\bbN$,
    \begin{align*}
        |\bbS_R(T,\emptyset)[m]|=|\bbS_R(T\cup\{p\},\emptyset)[m]|+|\bbS_R(T\cup\{p\},\emptyset)[nm]|.
    \end{align*}
    Let $\epsilon>0$. Pick $M_1>0$ such that if $n>M_1$ then
    \begin{align}\label{i1}
        \frac{\epsilon}{4}&>\left|\frac{|\bbS_R(T,\emptyset)[\lfloor n/N(p)\rfloor]|}{\left|F_{\lfloor n/N(p)\rfloor}\right|}-\delta\right|=\left|\frac{|\bbS_R(T\cup\{p\},\emptyset)[\lfloor n/N(p)\rfloor]|}{\left|F_{\lfloor n/N(p)\rfloor}\right|}+\frac{|\bbS_R(T\cup\{p\},\emptyset)[n]|}{\left|F_{\lfloor n/N(p)\rfloor}\right|}-\delta\right|.
    \end{align}
    Now, pick $k\in2\bbN$ such that $\frac{1}{N(p)^k}<\frac{\epsilon}{4}$. By Lemmas \ref{rf} and \ref{ro}, $\forall j\in\{1,\dots,k\}$, $\exists M_2>0$ such that if $n>M_2$ then
    \begin{align}\label{ineq1}
        \left|\frac{\left|F_{\lfloor n/N(p)^j\rfloor}\right|}{\left|F_{N(p)^j\lfloor n/N(p)^k\rfloor}\right|}-\frac{1}{N(p)^j}\right|<\frac{\epsilon}{2^{3+k-j}}.
    \end{align}
    For $n>M:=\max\{M_1,M_2\}$, 
    \begin{align}
    \begin{matrix}
        |\bbS_R(T\cup\{p\},\emptyset)[\lfloor n/N(p)^k\rfloor]|\leq \left|F_{\lfloor n/N(p)^k\rfloor}\right|=\frac{\left|F_{\lfloor n/N(p)^k\rfloor}\right|}{|F_n|}|F_n|\\
        \leq \frac{\left|F_{\lfloor n/N(p)^k\rfloor}\right|}{\left|F_{N(p)^k\lfloor n/N(p)^k\rfloor}\right|}|F_n|<\left(\frac{1}{N(p)}+\frac{\epsilon}{8}\right)|F_n|<\frac{\epsilon}{4}|F_n|.\hspace{45pt}
    \end{matrix}
        \label{ineq2}
    \end{align}
    and we also have
    \begin{align}
    \begin{matrix}
        \left|-\delta|F_n|\sum_{j=1}^k\left(-\frac{1}{N(p)}\right)^j-\delta|F_n|\frac{1}{N(p)+1}\right|=\delta|F_n|\left|\frac{-\frac{1}{N(p)}-\left(-\frac{1}{N(p)}\right)^{k+1}}{1+\frac{1}{N(p)}}+\frac{1}{N(p)+1}\right|\\
        =\delta|F_n|\left|\frac{N(p)^{-k}}{N(p)+1}\right|<\delta|F_n|N(p)^{-k}<\frac{\epsilon}{4}\delta|F_n|\leq\frac{\epsilon}{4}|F_n|.
    \end{matrix}
        \label{ineq3}
    \end{align}
    If $n>N(p)^kM$, then for all $0\leq j\leq k$, we have $n/N(p)^j>M$, and hence using inequality \ref{i1}, 
    \begin{align}
        \footnotesize\begin{matrix}
            \left|(-1)^j|\bbS_R(T\cup\{p\},\emptyset)[\lfloor n/N(p)^j\rfloor]|+(-1)^j|\bbS_R(T\cup\{p\},\emptyset)[\lfloor n/N(p)^{j+1}\rfloor]|-(-1)^j\delta\left|F_{\lfloor n/N(p)^{j+1}\rfloor}\right|\right|\\
        =\left|F_{\lfloor n/N(p)^{j+1}\rfloor}\right|\left|(-1)^j\frac{|\bbS_R(T\cup\{p\},\emptyset)[\lfloor n/N(p)^j\rfloor]|}{\left|F_{\lfloor n/N(p)^{j+1}\rfloor}\right|}+(-1)^j\frac{|\bbS_R(T\cup\{p\},\emptyset)[\lfloor n/N(p)^{j+1}\rfloor]|}{\left|F_{\lfloor n/N(p)^{j+1}\rfloor}\right|}-(-1)^j\delta\right|\hspace{20pt}\\
        \leq\left|F_{\lfloor n/N(p)^{j+1}\rfloor}\right|\left|(-1)^j\frac{|\bbS_R(T\cup\{p\},\emptyset)[\lfloor n/N(p)^{j+1}\rfloor N(p)]|}{\left|F_{\lfloor n/N(p)^{j+1}\rfloor}\right|}+(-1)^j\frac{|\bbS_R(T\cup\{p\},\emptyset)[\lfloor n/N(p)^{j+1}\rfloor]|}{\left|F_{\lfloor n/N(p)^{j+1}\rfloor}\right|}-(-1)^j\delta\right|\\
        <\frac{\epsilon}{4}\left|F_{\lfloor n/N(p)^{j+1}\rfloor}\right|.\hspace{300pt}
        \end{matrix}\normalsize\label{ineq4}
    \end{align}
    Using the triangle inequality and combining the inequalities \ref{ineq1}, \ref{ineq2}, \ref{ineq3}, and \ref{ineq4}, we obtain:
    \begin{align*}
        &\left|\frac{|\bbS_R(T\cup\{p\},\emptyset)[n]|}{|F_n|}-\frac{1}{N(p)+1}\delta\right|<\frac{\epsilon}{4}\left(\sum_{j=1}^k\frac{\left|F_{\lfloor n/N(p)^j\rfloor}\right|}{|F_n|}\right)+\frac{\epsilon}{4}+\frac{\epsilon}{4}\\
        &< \frac{\epsilon}{4}\left(\sum_{j=1}^k\frac{1}{N(p)^j}+\frac{\epsilon}{2^{3+k-j}}\right)+\frac{\epsilon}{2}<\frac{\epsilon}{4}\left(\frac{\frac{1}{N(p)}-\frac{1}{N(p)^{k+1}}}{1-\frac{1}{N(p)}}+3\right)=\frac{\epsilon}{4}\left(\underbrace{\frac{1-\frac{1}{N(p)^k}}{N(p)-1}}_{<1}+3\right)<\epsilon.
    \end{align*}
    Hence, $\bbS_R(T\cup\{p\},\emptyset)$ has density $\frac{1}{N(p)+1}\delta$.
\end{proof}

\subsection{Proof of Theorem \ref{1}}
\begin{itemize}
    \item[(i)]  When $P$ is finite.
    
    \noindent Applying Lemma \ref{u}, with $P=\emptyset$ and $A=P$, and Lemma \ref{s}, we get
    \begin{align*}
        d(\bbS_R(T,P))&\underset{\text{Lemma \ref{u}}}{=}N(m_P)d(\bbS_R(T\cup P,\emptyset))\\
        &\underset{\text{Lemma \ref{s}}}{=}d(\bbS_R)\prod_{p\in P}N(p)\prod_{q\in T\cup P}\frac{1}{N(q)+1}\\
        &=d(\bbS_R)\prod_{t\in T}\frac{1}{N(t)+1}\prod_{p\in P}\frac{N(p)}{N(p)+1}.
    \end{align*}

\item[(ii)] When $P$ is infinite.

\noindent Let $\{p_1,p_2,\dots\}\subseteq P$ with $N(p_1)\leq N(p_2)\leq \cdots$. Since $\frac{N(p_k)}{N(p_k)+1}<1$ for all $k$, it follows that $\left\{\prod_{i=1}^n\frac{N(p_k)}{N(p_k)+1}\right\}_{n\in\bbN}$ form a strictly decreasing sequence bounded below by $0$, and thus $\prod_{k=1}^n\frac{N(p_k)}{N(p_k)+1}$ converges, say $\alpha$. If $\alpha\neq0$, then
    \begin{align*}
        \sum_k\log\left(\frac{N(p_k)}{N(p_k)+1}\right)=\sum_k\underbrace{(\log(N(p_k))-\log(N(p_k)+1))}_{\in\left(\frac{1}{2N(p_k)},\frac{1}{N(p_k)}\right)}\leq\sum_k\frac{1}{N(p_k)}
    \end{align*}
    converges absolutely to $\log(\alpha)$, so that $\prod_k\frac{N(p_k)}{N(p_k)+1}$ is unchanged under arrangements.  The infinite product $\prod_k\frac{N(p_k)}{N(p_k)+1}$ is also unchanged under arrangements when $\alpha=0$. Therefore, $\prod_k\frac{N(p_k)}{N(p_k)+1}$ is well-defined.

    \noindent When $T=\emptyset$, by applying Tannery's Theorem\cite{h} and Lemma \ref{s}, we get
    \begin{align*}
        &d(\bbS_R\setminus \bbS_R(\emptyset,P))=\lim_{n\to\infty}\frac{(\bbS_R\setminus \bbS_R(\emptyset,P))[n]}{n}=\lim_{n\to\infty}\sum_{k=1}^\infty\frac{\bbS_R(\{p_k\},\{p_1,\dots,p_{k-1}\})[n]}{n}\\
        &\underset{\text{Tannery}}{=}\sum_{k=1}^\infty\lim_{n\to\infty}\frac{\bbS_R(\{p_k\},\{p_1,\dots,p_{k-1}\})[n]}{n}=\sum_{k=1}^\infty d(\bbS_R(\{p_k\},\{p_1,\dots,p_{k-1}\}))\\
        &\underset{\text{Lemma \ref{s}}}{=}\sum_{k=1}^\infty \left(d(\bbS_R)\frac{1}{N(p_k)+1}\prod_{i<k}\frac{N(p_i)}{N(p_i)+1}\right)\\
        &=\sum_{k=1}^\infty \left(d(\bbS_R)\left(1-\frac{N(p_k)}{N(p_k)+1}\right)\prod_{i<k}\frac{N(p_i)}{N(p_i)+1}\right)\\
        &=d(\bbS_R)\lim_{N\to\infty}\sum_{k=1}^N\left(\prod_{i<k}\frac{N(p_i)}{N(p_i)+1}-\prod_{i<k+1}\frac{N(p_i)}{N(p_i)+1}\right)\\
        &=d(\bbS_R)\lim_{N\to\infty}\left(1-\prod_{i\leq N}\frac{N(p_i)}{N(p_i)+1}\right)=d(\bbS_R)(1-\alpha).
    \end{align*}
    This implies that
    \begin{align*}
        d(\bbS_R(\emptyset,P))=d(\bbS_R)-d(\bbS_R\setminus \bbS_R(\emptyset,P))=d(\bbS_R)\alpha=d(\bbS_R)\prod_{p\in P}\frac{N(p)}{N(p)+1}.
    \end{align*}
    Applying Lemma \ref{u} with $T=\emptyset$ and $S=T$, we get
    \begin{align*}
        d(\bbS_R(T,P))&=\frac{d(\bbS_R(\emptyset,T\cup P))}{N(m_T)}=d(\bbS_R)\prod_{t\in T}\frac{1}{N(t)}\prod_{t\in T}\frac{N(t)}{N(t)+1}\prod_{p\in P}\frac{N(p)}{N(p)+1}\\
        &=d(\bbS_R)\prod_{t\in T}\frac{1}{N(t)+1}\prod_{p\in P}\frac{N(p)}{N(p)+1}.\qed
    \end{align*} 
\end{itemize}

\section{Proof of Theorem \ref{2}}
\subsection{Auxiliary Lemmas for Theorem \ref{2}}
\begin{lemma}\label{rd}
    For any $r,m\in R$ with $m\neq0$,
    \begin{align*}
        d(r+mR)=\frac{1}{N(m)}.
    \end{align*}
\end{lemma}
\begin{proof}
    Let $r,m\in R$ with $m\neq0$.
    \begin{itemize}
        \item[(i)] When $R=\bbF_q[t]$; suppose $n>q^{\deg(r)}$. If $x\in F_{nq^{-\deg(m)}}$, since $n>q^{\deg(r)},q^{\deg(mx)}$, then 
        \begin{align*}
            q^{\deg(r+mx)}\leq q^{\max\{\deg(r),\deg(mx)\}}<n,
        \end{align*}
        Thus, $r+mF_{nq^{-\deg(m)}}\subseteq F_n$, so $r+mF_{nq^{-\deg(m)}}\subseteq (r+m\bbF_q[t])\cap F_n$. If $r+mx\in F_n$, since $n>q^{\deg(r)}$, then $mx\in F_n$ and
        \begin{align*}
            q^{\deg(x)}=q^{\deg(mx)}q^{-\deg(m)}<nq^{-\deg(m)},
        \end{align*}
        so $x\in F_{nq^{-\deg(m)}}$. Thus, $|(r+m\bbF_q[t])\cap F_n|\leq |F_{nq^{-\deg(m)}}|$. Then we obtain
        \begin{align*}
            \frac{|F_{nq^{-\deg(m)}}|}{|F_n|}=\frac{|r+mF_{nq^{-\deg(m)}}|}{|F_n|}\leq \frac{|(r+m\bbF_q[t])\cap F_n|}{|F_n|}\leq\frac{|F_{nq^{-\deg(m)}}|}{|F_n|}.
        \end{align*}
        That is, $\frac{|(r+m\bbF_q[t])\cap F_n|}{|F_n|}=\frac{|F_{nq^{-\deg(m)}}|}{|F_n|}$ if $n>q^{\deg(r)}$. By Lemma \ref{rf}, 
        \begin{align*}
            d(r+m\bbF_q[t])=\lim_{n\to\infty}\frac{|(r+m\bbF_q[t])\cap F_n|}{|F_n|}=\frac{1}{q^{\deg(m)}}=\frac{1}{N(m)}.
        \end{align*}

        \item[(ii)] When $R=\cO_D$. Extend $F_n$, originally defined for $n\in\bbN$, to $n\in\bbR_{>0}$. 
        
        \noindent If $x=x_1+\sqrt{-D}x_2\in F_{(\sqrt{n}-\sqrt{N(r)})^2N(m)^{-1}}$, then
        \begin{align*}
            N(r+mx)&\leq \left(\sqrt{N(r)}+\sqrt{N(mx)}\right)^2=\left(\sqrt{N(r)}+\sqrt{N(m)}\sqrt{N(x)}\right)^2\\
            &<\left(\sqrt{N(r)}+\sqrt{n}-\sqrt{N(r)}\right)^2=n.
        \end{align*}
        Thus, $r+mF_{(\sqrt{n}-\sqrt{N(r)})^2N(m)^{-1}}\subseteq F_n$. If $r+mx\in F_n$, then
        \begin{align*}
            N(x)&=N(mx)N(m)^{-1}=N((r+mx)-r)N(m)^{-1}\\
            &\leq \left(\sqrt{N(r+mx)}+\sqrt{N(r)}\right)^2N(m)^{-1}\\
            &<(\sqrt{n}+\sqrt{N(r)})^2N(m)^{-1},
        \end{align*}
        so $x\in F_{(\sqrt{n}+\sqrt{N(r)})^2N(m)^{-1}}$. This implies that $|(r+m\cO_D)\cap F_n|\leq|F_{(\sqrt{n}+\sqrt{N(r)})^2N(m)^{-1}}|$. Then we obtain
        \begin{align*}
            \frac{|F_{(\sqrt{n}-\sqrt{N(r)})^2N(m)^{-1}}|}{|F_n|}&=\frac{|r+mF_{(\sqrt{n}-\sqrt{N(r)})^2N(m)^{-1}}|}{|F_n|}\\
            &\leq\frac{|(r+m\cO_D)\cap F_n|}{|F_n|}\leq\frac{|F_{(\sqrt{n}+\sqrt{N(r)})^2N(m)^{-1}}|}{|F_n|}.
        \end{align*}
        Using Lemma \ref{la}, one can show 
        \begin{align*}
            \lim_{n\to\infty}\frac{|F_{(\sqrt{n}-\sqrt{N(r)})^2N(m)^{-1}}|}{|F_n|}=\lim_{n\to\infty}\frac{|F_{(\sqrt{n}+\sqrt{N(r)})^2N(m)^{-1}}|}{|F_n|}=\frac{1}{N(m)}.
        \end{align*}
        Hence, $d(r+m\cO_D)=\frac{1}{N(m)}$.
    \end{itemize}
\end{proof}
\begin{lemma}\label{eq}
    Let $P\subseteq\bbI_R$ be infinite. Then $\sum_{p\in P}\frac{1}{N(p)}=\infty$ if and only if $\prod_{p\in P}\frac{N(p)}{1+N(p)}=0$. Moreover, if $\sum_{p\in P}\frac{1}{N(p)}\leq L<\infty$ for $L\in\bbR$, then $\prod_{p\in P}\frac{N(p)}{1+N(p)}\geq e^{-L}$.
\end{lemma}
\begin{proof}
    It follows from
    \begin{align*}
        \frac{1}{2}\sum_{p\in P}\frac{1}{N(p)}\leq-\log\prod_{p\in P}\frac{N(p)}{1+N(P)}=\sum_{p\in P}\log\frac{1+N(p)}{N(p)}\leq\sum_{p\in P}\frac{1}{N(p)}.
    \end{align*}
\end{proof}

\begin{lemma}\label{app}
    Let $A\subseteq R\setminus\{0\}$ with $d(A)>0$. Then 
    \begin{align*}
        \sum_{p\in A}\frac{1}{N(p)}=\infty.
    \end{align*}
\end{lemma}
\begin{proof} Pick $M_0\in\bbN$ such that 
    \begin{align*}
        \frac{d(A)}{2}\leq\frac{|F_n\cap A|}{|F_n|}\leq \frac{3d(A)}{2}\quad\forall n\geq M_0.
    \end{align*}
If $k\geq\log_4M_0$, then
        \begin{align*}
            |A\cap(F_{4^{k+1}}\setminus F_{4^k})|&=|A\cap F_{4^{k+1}}|-|A\cap F_{4^k}|\geq\frac{d(A)}{2}|F_{4^{k+1}}|-\frac{3d(A)}{2}|F_{4^k}|.
        \end{align*}
    \begin{itemize}
        \item[(i)] When $R=\bbF_q[t]$; for each $k\in\bbN$, since $q\geq2$, or $q^4\geq16$,
        \begin{align*}
            \frac{\frac{d(A)}{2}|F_{4^{k+1}}|-\frac{3d(A)}{2}|F_{4^k}|}{q^{4^{k+1}}}=\frac{\frac{d(A)}{2}q^{4^{k+1}}-\frac{3d(A)}{2}q^{4^k}}{q^{4^{k+1}}}=\frac{d(A)}{2}\left(1-\frac{3}{q^4}\right)\geq\frac{d(A)}{8}.
        \end{align*}
        Then we obtain
        \begin{align*}
            \sum_{a\in A}\frac{1}{q^{\deg(a)}}&\geq\sum_{k=\left\lceil\log_4M_0\right\rceil}^\infty\sum_{a\in A\cap(F_{4^{k+1}}\setminus F_{4^k})}\frac{1}{q^{4^{k+1}}}\\
            &=\sum_{k=\left\lceil\log_4M_0\right\rceil}^\infty\frac{|A\cap(F_{4^{k+1}}\setminus F_{4^k})|}{q^{4^{k+1}}}\geq\sum_{k=\left\lceil\log_4M_0\right\rceil}^\infty\frac{d(A)}{8}=\infty.
        \end{align*}
        
        \item[(ii)] When $R=\cO_D$; pick $M\in\bbN$ with $M\geq\log_4M_0$ such that $\sqrt{2}\cdot4^M>4^M+\sqrt{1+D}$. 
        
        \noindent For $k\geq M$, by Proposition \ref{la},
        \begin{align*}
            \frac{\frac{\overline{d}(A)}{2}|F_{4^{k+1}}|-\frac{3\overline{d}(B)}{2}|F_{4^k}|}{|4^{k+1}|^2}&\geq\frac{\frac{\overline{d}(B)}{2}\frac{\pi(4^{k+1}-\sqrt{1+D})^2}{\sqrt{D}}-\frac{3\overline{d}(B)}{2}\frac{\pi(4^k+\sqrt{1+D})^2}{\sqrt{D}}}{16^{k+1}}\\
            &\geq\frac{\frac{\overline{d}(B)}{2}\frac{\pi(4^{k+1}-4^k)^2}{\sqrt{D}}-\frac{3\overline{d}(B)}{2}\frac{\pi(\sqrt{2}\cdot4^k)^2}{\sqrt{D}}}{16^{k+1}}=\frac{3\pi\overline{d}(B)}{32\sqrt{D}}.
        \end{align*}
        Then we obtain
        \begin{align*}
            \sum_{b\in B}\frac{1}{|b|^2}&\geq\sum_{k=M}^\infty\sum_{b\in B\cap(F_{4^{k+1}}\setminus F_{4^k})}\frac{1}{|4^{k+1}|^2}\\
            &=\sum_{k=M}^\infty\frac{|B\cap(F_{4^{k+1}}\setminus F_{4^k})|}{16^{k+1}}\geq\sum_{k=M}^\infty\frac{3\pi\overline{d}(B)}{32\sqrt{D}}=\infty.
        \end{align*}
    \end{itemize}
\end{proof}

\subsection{Proof of Theorem \ref{2}} Let $P=r+mR$. The density of the set of squarefree elements divisible by no prime in $P$ is 
\begin{align*}
    d(\bbS_R(\emptyset,P))=d(\bbS_R)\prod_{p\in P}\frac{N(p)}{1+N(p)}
\end{align*}
by Theorem \ref{1}. Since $d(P)=d(r+mR)=\frac{1}{N(m)}>0$ by  Lemma \ref{rd}, it follows from Lemma \ref{app} that $\sum_{p\in P}\frac{1}{N(p)}=\infty$. Hence, $\prod_{p\in P}\frac{N(p)}{1+N(p)}=0$ by Lemmas \ref{eq}. \qed

\section{Mersenne Irreducibles} The \emph{Mersenne prime} in $\bbN$ is of the form 
\begin{align*}
    p=2^n-1=\sum_{k=0}^{n-1}2^k.
\end{align*}
\begin{definition}[Mersenne irreducible] Let $p\in\bbI_R$.
    \begin{itemize}
        \item $p\in\bbI_{\bbF_q[t]}$ is a \emph{Mersenne irreducible} if $p=\sum_{k=0}^{n-1}t^k$ for some $n\in\bbN$. 

        \item $p\in\bbI_{\cO_D}$ is a \emph{Mersenne irreducible} if $p=\sum_{k=0}^{n-1}(1+\sqrt{-D})^k$ for some $n\in\bbN$. 
    \end{itemize}
    Let $\cM_R$ denote the set of Mersenne irreducibles.
\end{definition}
For any $P\subseteq\bbI_R$ and any $M\in\bbN$, let
\begin{align*}
    P_M&=\begin{cases}
        \left\{p\in P:p=\sum_{k=0}^{l-1}t^k\text{ for some }l\leq M\right\}&\text{if $R=\bbF_q[t]$}\\
        \left\{p\in P:\sum_{k=0}^{l-1}(1+\sqrt{-D})^k\text{ for some }l\leq M\right\}&\text{if $R=\cO_D$}
    \end{cases}\subseteq\cM_R;\\
    A_M&=\prod_{p\in P_M}\frac{N(p)}{1+N(p)}.
\end{align*}
\begin{lemma}\label{p2}
    Let $a\in\bbR$. Let $P=\{p_1,p_2,\dots\}\subseteq\bbI_R$ be infinite. Suppose that there are $M\in\bbN$ and $b,\rho\in\bbR_{>0}$ with $\rho>1$ such that 
    \begin{align*}
        N(p_k)\geq b\rho^k-a\quad\forall k\geq M.
    \end{align*}
    Then the set of squarefree elements not divisible by any element of $P$ has positive density.
\end{lemma}
\begin{proof}
    Since $N(p_k)\geq b\rho^k-a$ for all $k\geq M$,
    \begin{align*}
        \sum_{p\in P}\frac{1}{N(p)}=\sum_{k=1}^\infty\frac{1}{N(p_k)}\leq \sum_{k=1}^{M-1}\frac{1}{N(p_k)}+\sum_{k=M}^\infty\frac{1}{b\rho^{k}-a}=L<\infty,
    \end{align*}
    and by Lemma \ref{eq} and Theorem \ref{1}, it follows that 
    \begin{align*}
        d(\bbS_R(\emptyset,P))=\prod_{p\in P}\frac{N(P)}{1+N(p)}\geq e^{-L}>0.
    \end{align*}
    Hence, the set of squarefree elements not divisible by any element of $P$ has positive density.
\end{proof}
\begin{remark}
    Lemma \ref{p2} implies that if $P\subseteq\cM_R$ is infinite, then $d(\bbS_R(\emptyset,P))>0$.
    \begin{itemize}
        \item Every $p\in\cM_{\bbF_q[t]}$ has the form $p=\sum_{k=0}^{n-2}t^k$ for some $n\in\bbN$. Then $\deg(p)=n-1$, so $N(p)=q^{n-1}=\frac{1}{q}q^{n-1}$.

        \item Every $p\in\cM_{\cO_D}$ has the form 
        \begin{align*}
            p=\sum_{k=0}^{n-1}(1+\sqrt{-D})^k=\frac{(1+\sqrt{-D})^n-1}{\sqrt{-D}}
        \end{align*}
        for some $n\in\bbN$. Pick $M\in\bbN$ such that $(1+D)^{\frac{M}{2}}-1\geq (1+D)^{\frac{M-1}{2}}$. For $n\geq M$,
            \begin{align}
            \begin{matrix}
                N(p)=N\left(\frac{(1+\sqrt{-D})^n-1}{\sqrt{-D}}\right)=\frac{N((1+\sqrt{-D})^n-1)}{N(\sqrt{-D})}\hspace{50pt}\\
            \geq\frac{\left(\sqrt{N((1+\sqrt{-D})^n)}-\sqrt{N(1)}\right)^2}{D}=\frac{(\sqrt{(1+D)^n}-1)^2}{D}\\
            \geq\frac{(1+D)^{n-1}}{D}=\frac{1}{D(1+D)}(1+D)^n,\hspace{55pt}
            \end{matrix}\label{ub}
        \end{align}
        using the reverse triangle inequality $|x-y|\geq||x|-|y||$.
    \end{itemize}
\end{remark}
\begin{theorem}\label{3}
    For any $P\subseteq\bbI_R$ and $M\in\bbN$, one has
    \begin{align*}
        A_M\geq\prod_{p\in P}\frac{N(p)}{1+N(p)}\geq \begin{cases}
            A_M\exp\left(\frac{1}{(q-1)q^{M-1}}\right)&\text{if $R=\bbF_q[t]$};\\
            A_M\exp\left(\frac{1}{(1+D)^{M-1}}\right)&\text{if $R=\cO_D$}.
        \end{cases}
    \end{align*}
\end{theorem}
\begin{proof} Note that
    \begin{align*}
        A_M&\geq\prod_{p\in P}\frac{N(p)}{1+N(p)}=\prod_{p\in P_M}\frac{N(p)}{1+N(p)}\prod_{p\in P\setminus P_M}\frac{N(p)}{1+N(p)}\geq A_M\prod_{p\in P\setminus P_M}\frac{N(p)}{1+N(p)}.
    \end{align*}
    So, it suffices to show that for any $M\in\bbN$,
    \begin{align*}
        \prod_{p\in P\setminus P_M}\frac{N(p)}{1+N(p)}\geq \begin{cases}
            \exp\left(\frac{1}{(q-1)q^{M-1}}\right)&\text{if $R=\bbF_q[t]$};\\
            \exp\left(\frac{1}{(1+D)^{M-1}}\right)&\text{if $R=\cO_D$}.
        \end{cases}
    \end{align*}
    \begin{itemize}
        \item When $R=\bbF_q[t]$.
        \begin{align*}
        \sum_{p\in P\setminus P_M}\frac{1}{N(p)}&\leq\sum_{n>M}\frac{1}{N\left(\sum_{k=0}^{n-1}t^k\right)}=\sum_{n>M}\frac{1}{q^{n-1}}=\frac{\frac{1}{q^M}}{1-\frac{1}{q}}=\frac{1}{(q-1)q^{M-1}}.
    \end{align*}
    By Lemma \ref{eq}, we obtain 
    \begin{align*}
        \prod_{p\in P_M}\frac{N(p)}{1+N(p)}\geq \exp\left(\frac{1}{(q-1)q^{M-1}}\right).
    \end{align*}
    
        \item When $R=\cO_D$; using the inequality \ref{ub}: $N(p)\geq\frac{(1+D)^{n-1}}{D}$, we obtain
        \begin{align*}
        \sum_{p\in P\setminus P_M}\frac{1}{N(p)}&\leq\sum_{n>M}\frac{D}{(1+D)^{n-1}}=\frac{\frac{D}{(1+D)^{M}}}{1-\frac{1}{1+D}}=\frac{1}{(1+D)^{M-1}}.
    \end{align*}
    By Lemma \ref{eq}, we obtain
     \begin{align*}
        \prod_{p\in P\setminus P_M}\frac{N(p)}{1+N(p)}\geq \exp\left(\frac{1}{(1+D)^{M-1}}\right).
    \end{align*}
    \end{itemize}
\end{proof}
Theorem \ref{3} allows us to approximate the density of $\bbS_R(\emptyset,P)$.
   \quad



\begin{thebibliography}{9}
\bibitem{b} Brown, R. \emph{The natural density of some sets of square-free numbers}. Integers, \textbf{21}, 2021.

\bibitem{cv} Cellarosi, F., Vinogradov, I. \emph{Ergodic properties of $k$-free integers in number fields}. Journal of Modern Dynamics, \textbf{7}(3), 461--488, 2013.


\bibitem{c} Chandrasekharan, K. \emph{Introduction to Analytic Number Theory}. Springer Science \& Business Media, 2012.

\bibitem{ha} Halberstam, H. \emph{Gaps in integer sequences}, Mathematics Magazine, \textbf{56}(3), 131--140, 1983.

\bibitem{h} Hofbauer, J. \emph{A simple proof of and related identities}. The American mathematical monthly, \textbf{109}(2), 196--200, 2002.

\bibitem{j} Jameson, G. J. O. \emph{Even and odd square-free numbers}, The Mathematical Gazette \textbf{94}, 123--127, 2010.


\bibitem{nj} Neukirch, J. Algebraic number theory. Springer Science \& Business Media, \textbf{322}, 2013.

\bibitem{r} Rosen, M. \emph{Number theory in function fields}.  Springer Science \& Business Media, \textbf{210}, 2013.



\end{thebibliography}
\end{document}